\documentclass{amsart}
\usepackage{graphicx}
\usepackage[all]{xy}
\usepackage{amsthm}
\usepackage{amssymb}
\usepackage{cite}
\usepackage{mathtools}

\usepackage{amsmath,amsfonts}

\theoremstyle{definition}
\newtheorem{1def}{Definition}[section]

\theoremstyle{plain}
\newtheorem{thm}[1def]{Theorem}
\newtheorem{lem}[1def]{Lemma}
\newtheorem{pro}[1def]{Proposition}
\newtheorem{cor}[1def]{Corollary}

\theoremstyle{remark}

\numberwithin{equation}{section}

\begin{document}
\title{Some $q$-exponential formulas for finite-dimensional $\square_q$-modules}
\author{Yang Yang}
\address{Department of Mathematics, University of Wisconsin, Madison, WI 53706, USA}
\email{yyang@math.wisc.edu}

\begin{abstract}
We consider the algebra $\square_q$ which is a mild generalization of the quantum algebra $U_q(\frak{sl}_2)$. The algebra $\square_q$ is defined by generators and relations. The generators are $\{x_i\}_{i\in \mathbb{Z}_4}$, where $\mathbb{Z}_4$ is the cyclic group of order $4$. For $i\in \mathbb{Z}_4$ the generators $x_i$,$x_{i+1}$ satisfy a $q$-Weyl relation, and $x_i$,$x_{i+2}$ satisfy a cubic $q$-Serre relation. For $i\in \mathbb{Z}_4$ we show that the action of $x_i$ is invertible on every nonzero finite-dimensional $\square_q$-module. We view $x_i^{-1}$ as an operator that acts on nonzero finite-dimensional $\square_q$-modules. For $i\in \mathbb{Z}_4$, define $\mathfrak{n}_{i,i+1}=q(1-x_ix_{i+1})/(q-q^{-1})$. We show that the action of $\mathfrak{n}_{i,i+1}$ is nilpotent on every nonzero finite-dimensional $\square_q$-module. We view the $q$-exponential ${\rm {exp}}_q(\mathfrak{n}_{i,i+1})$ as an operator that acts on nonzero finite-dimensional $\square_q$-modules. In our main results, for $i,j\in \mathbb{Z}_4$ we express each of ${\rm {exp}}_q(\mathfrak{n}_{i,i+1})x_j{\rm {exp}}_q(\mathfrak{n}_{i,i+1})^{-1}$ and ${\rm {exp}}_q(\mathfrak{n}_{i,i+1})^{-1}x_j{\rm {exp}}_q(\mathfrak{n}_{i,i+1})$ as a polynomial in $\{x_k^{\pm 1}\}_{k\in \mathbb{Z}_4}$.\\


\end{abstract}
\maketitle
\section{Introduction}
This paper is about a certain algebra $\square_q$; we will recall the definition shortly. Broadly speaking it can be viewed as a generalization of the quantum algebra $U_q(\frak{sl}_2)$. In order to motivate our results we make some comments about $U_q(\frak{sl}_2)$. We will work with the equitable presentation, which was introduced in \cite{equit} and investigated further in \cite{alnajjar, neubauer, huang, irt, uqsl2hat, nonnil, qtet, tersym, uawe,fduq,boyd,paul}. Let $\mathbb{F}$ denote an algebraically closed field. Fix $0\ne q\in \mathbb{F}$ that is not a root of unity. In the equitable presentation, the $\mathbb{F}$-algebra $U_q(\frak{sl}_2)$ has generators $x, y^{\pm 1}
, z$ and relations $yy^{-1}=y^{-1}y=1$,
\begin{equation*}
\frac{qxy-q^{-1}yx}{q-q^{-1}}=1,\quad\frac{qyz-q^{-1}zy}{q-q^{-1}}=1,\quad\frac{qzx-q^{-1}xz}{q-q^{-1}}=1.
\end{equation*}
Define
\begin{equation*}
n_x=\frac{q(1-yz)}{q-q^{-1}},\quad n_y=\frac{q(1-zx)}{q-q^{-1}},\quad  n_z=\frac{q(1-xy)}{q-q^{-1}}.
\end{equation*}

On every nonzero finite-dimensional $U_q(\frak{sl}_2)$-module, $x, y, z$ are invertible (see \cite[Lemma~5.15]{luz}) and $n_x, n_y, n_z$ are nilpotent (see \cite[Lemma~5.14]{luz}). Recall from \cite[p.~204]{t1} the $q$-exponential function
\begin{equation*}
{\rm {exp}}_q (T)=\sum_{n\in\mathbb{N}}\frac{q^{n \choose 2 }}{[n]_q^!}T^n.
\end{equation*}

In \cite[Sections~5,\,6]{equit} it was shown that the following equations hold on every nonzero finite-dimensional $U_q(\frak{sl}_2)$-module:
\begin{subequations}
\begin{align}
{\rm {exp}}_q (n_x)x\,{\rm {exp}}_q (n_x)^{-1}=x+z-z^{-1},\\
{\rm {exp}}_q (n_x)y\,{\rm {exp}}_q (n_x)^{-1}=z^{-1},\\
{\rm {exp}}_q (n_x)z\,{\rm {exp}}_q (n_x)^{-1}=zyz,\\
{\rm {exp}}_q (n_x)^{-1}x\,{\rm {exp}}_q (n_x)=x+y-y^{-1},\\
{\rm {exp}}_q (n_x)^{-1}y\,{\rm {exp}}_q (n_x)=yzy,\\
{\rm {exp}}_q (n_x)^{-1}z\,{\rm {exp}}_q (n_x)=y^{-1}.
\end{align}
\end{subequations}

Cyclically permuting $x,y,z$ in the above equations, we get $12$ more equations. Our goal in this paper is to find analogous equations that apply to $\square_q$.

We now discuss the algebra $\square_q$. This algebra was introduced in \cite[Definition~5.1]{p1}. We mention some algebras that are related to $\square_q$. For the positive part $U_q^{+}(\widehat {\frak {sl}}_2)$ of $U_q(\widehat {\frak {sl}}_2)$ (see \cite[\rm Definition~1.1]{nonnil}), there exists an injective algebra homomorphism from $U_q^{+}(\widehat {\frak {sl}}_2)$ to $\square_q$ (see \cite[Proposition~5.5]{p1}). For the $q$-Onsager algebra $\mathcal{O}_q$ (see \cite[\rm Section~2]{bas}), there exists an injective algebra homomorphism from $\mathcal{O}_q$ to $\square_q$ (see \cite[Proposition~11.9]{p1}). For the quantum loop algebra $U_q(L(\frak{sl}_2))$, there exists an injective algebra homomorphism from $\square_q$ to $U_q(L(\frak{sl}_2))$ (see \cite[Proposition~5.5]{p1} and \cite[Propostions~4.1, 4.3]{miki}). For the $q$-tetrahedron algebra $\boxtimes_q$ (see \cite[\rm Definition~6.1]{qtet}), there exists an injective algebra homomorphism from $\square_q$ to $\boxtimes_q$ (see \cite[Proposition~5.5]{p1} and \cite[Propostions~4.1]{miki}).

The $\mathbb{F}$-algebra $\square_q$ is defined as follows (formal definitions start in Section 2). The generators are $\{x_i\}_{i\in \mathbb{Z}_4}$, where $\mathbb{Z}_4=\mathbb{Z}/4\mathbb{Z}$ is the cyclic group of order $4$. The relations are
\begin{gather*}
\frac{qx_ix_{i+1}-q^{-1}x_{i+1}x_i}{q-q^{-1}}=1, \\
x_i^3x_{i+2}-[3]_qx_i^2x_{i+2}x_i+[3]_qx_ix_{i+2}x_i^2-x_{i+2}x_i^3=0,
\end{gather*}
for $i\in \mathbb{Z}_4$. We will state our main results after some preliminary remarks. We show that for $i\in \mathbb{Z}_4$ the action of $x_i$ is invertible on every nonzero finite-dimensional $\square_q$-module. We view $x_i^{-1}$ as an operator that acts on nonzero finite-dimensional $\square_q$-modules. For $i\in \mathbb{Z}_4$ define
\begin{equation*}
\mathfrak{n}_{i,i+1}=\frac{q(1-x_ix_{i+1})}{q-q^{-1}}.
\end{equation*}
We show that the action of $\mathfrak{n}_{i,i+1}$ is nilpotent on every nonzero finite-dimensional $\square_q$-module. We view the $q$-exponential ${\rm {exp}}_q(\mathfrak{n}_{i,i+1})$ as an operator that acts on nonzero finite-dimensional $\square_q$-modules. For $i,j\in \mathbb{Z}_4$ consider the two expressions
\begin{equation}
\label{equ:11}
{\rm {exp}}_q(\mathfrak{n}_{i,i+1})x_j{\rm {exp}}_q(\mathfrak{n}_{i,i+1})^{-1}, \qquad {\rm {exp}}_q(\mathfrak{n}_{i,i+1})^{-1}x_j{\rm {exp}}_q(\mathfrak{n}_{i,i+1}).
\end{equation}
For each expression in (\ref{equ:11}), expand both $q$-exponential terms. This yields a double sum with infinitely many terms. A natural question is, to what extent can this double sum be simplified? In our main results we will show that in fact, each double sum is a polynomial in $\{x_k^{\pm 1}\}_{k \in \mathbb{Z}_4}$. These results are Theorems {\ref {pro:81}}, {\ref {pro:82}} and Theorems {\ref {pro:vip1}}--{\ref {pro:vip4}}.

We mention another motivation for studying (\ref{equ:11}). Near the equation (1.1) we gave $18$ equations for $U_q(\frak{sl}_2)$. These equations were used to construct a rotator for $U_q(\frak{sl}_2)$ (see \cite[Definition~9.5]{luz}). These equations were also used to describe the Lusztig operators (see \cite{lus1,lus2}) for $U_q(\frak{sl}_2)$ in the equitable presentation (see \cite[Theorem~9.9]{luz}). We hope to obtain similar results for $\square_q$.

We mention a conceptual interest for finding a rotator of $\square_q$. Let $\rho$ denote the automorphism of $\square_q$ that sends $x_i\mapsto x_{i+1}$ for $i\in \mathbb{Z}_4$ (see Lemma {\ref{lem:m1}}). Let $V$ denote a finite-dimensional irreducible $\square_q$-module (see \cite[Definition~6.8]{yy}). Then the $\square_q$-modules $V$ and $V$ twisted via $\rho^2$ are isomorphic (see \cite[Corollary~1.7]{yy}). We hope that this isomorphism is given by the rotator in a canonical manner.

The paper is organized as follows. Section 2 contains the preliminaries. Section 3 contains some basic facts about $\square_q$. In Section 4 we describe some isomorphisms and antiisomorphisms for $\square_q$. In Section 5 we show that the action of each $x_i$ is invertible on every nonzero finite-dimensional $\square_q$-module. In Section 6 we show that the action of each $\mathfrak{n}_{i,i+1}$ is nilpotent on every finite-dimensional $\square_q$-module. In Section 7 we review the $q$-exponential function, and apply
it to $\mathfrak{n}_{i,i+1}$. In Sections 8 and 9 we prove our main results.

\section{Preliminaries}
Throughout the paper, we fix the following notation. Let $\mathbb{F}$ denote an algebraically closed field. Recall the set of natural numbers $\mathbb{N}=\{0,1,2,\dots\}$ and the ring of integers $\mathbb{Z}=\{0,\pm 1,\pm 2,\dots\}$. Let $\mathbb{Z}_4=\mathbb{Z}/4\mathbb{Z}$ denote the cyclic group of order $4$. We will be discussing algebras. An algebra is meant to be associative and
have a $1$.

Let $V$ denote a nonzero finite-dimensional vector space over $\mathbb{F}$. Let ${\rm{End}}(V)$ denote the $\mathbb{F}$-algebra consisting of the $\mathbb{F}$-linear maps from $V$ to $V$. An element $A\in {\rm{End}}(V)$ is called {\it nilpotent} whenever there exists a positive integer $n$ such that $A^n=0$. By an {\it eigenvalue} of $A$, we mean a root of the characteristic polynomial of $A$.

Fix $0\ne q\in \mathbb{F}$ such that $q$ is not a root of unity. For $n\in \mathbb{Z}$ define
\begin{equation*}
[n]_q = \frac{q^n-q^{-n}}{q-q^{-1}}.
\end{equation*}
For $n\in \mathbb{N}$ define
\begin{equation*}
[n]_q^! = \prod_{i=1}^{n}[i]_q.
\end{equation*}
We interpret $[0]_q^!=1$.

\section{The algebra $\square_q$}
In this section, we recall the algebra $\square_q$.

\begin{1def}
\label{def:box}
\cite[Definition~5.1]{p1}
Let $\square_q$ denote the $\mathbb{F}$-algebra with generators $\{x_i\}_{i\in \mathbb{Z}_4}$ and relations
\begin{gather}
\frac{qx_ix_{i+1}-q^{-1}x_{i+1}x_i}{q-q^{-1}}=1, \label{equ:1}\\
x_i^3x_{i+2}-[3]_qx_i^2x_{i+2}x_i+[3]_qx_ix_{i+2}x_i^2-x_{i+2}x_i^3=0. \label{equ:2}
\end{gather}
\end{1def}
The structure of the algebra $\square_q$ is analyzed in \cite{p1}. We don't need the full strength of the results in \cite{p1}, but we will use the following fact.
\begin{lem}
\label{lem:stu}
The elements $\{x_i\}_{i\in \mathbb{Z}_4}$ are linearly independent in $\square_q$.
\end{lem}
\begin{proof}
By \cite[Proposition~5.5]{p1}.
\end{proof}

We now give some formulas for later use.
\begin{lem}
\label{lem:c1}
For $i\in \mathbb{Z}_4$ and $n\in \mathbb{N}$ the following relations hold in $\square_q$:
\begin{gather}
{q^nx_i^nx_{i+1}-q^{-n}x_{i+1}x_i^n}=({q^n-q^{-n}})x_i^{n-1}, \label{equ:c1}\\
{q^nx_ix_{i+1}^n-q^{-n}x_{i+1}^nx_{i}}=({q^n-q^{-n}})x_{i+1}^{n-1}. \label{equ:c2}
\end{gather}
\end{lem}
\begin{proof}
Use ({\ref {equ:1}}) and induction on $n$.
\end{proof}

\begin{lem}
For $i\in \mathbb{Z}_4$ and $n\in \mathbb{N}$, the following relation holds in $\square_q$:
\begin{equation}
\begin{aligned}
\label{equ:c3}
x_i^{n}x_{i+2}&= \frac{[n-1]_q[n-2]_q}{[2]_q}x_{i+2}x_i^n-[n]_q[n-2]_qx_ix_{i+2}x_i^{n-1}\\
&+\frac{[n]_q[n-1]_q}{[2]_q}x_i^2x_{i+2}x_i^{n-2}.
\end{aligned}
\end{equation}
\end{lem}
\begin{proof}
Use ({\ref {equ:2}}) and induction on $n$.
\end{proof}

For $i\in \mathbb{Z}_4$ we define an element $\mathfrak{n}_{i,i+1}\in\square_q$. Later we will show that $\mathfrak{n}_{i,i+1}$ is nilpotent on every finite-dimensional $\square_q$-module.
\begin{lem}
For $i\in \mathbb{Z}_4$,
\begin{equation*}
q(1-x_ix_{i+1})=q^{-1}(1-x_{i+1}x_i).
\end{equation*}
\end{lem}
\begin{proof}
This equation is a reformulation of ({\ref {equ:1}}).
\end{proof}

\begin{1def}
\label{def:n}
For $i\in \mathbb{Z}_4$ define
\begin{equation}
\mathfrak{n}_{i,i+1}=\frac{q(1-x_ix_{i+1})}{q-q^{-1}}=\frac{q^{-1}(1-x_{i+1}x_i)}{q-q^{-1}}. \label{equ:3}
\end{equation}
\end{1def}

We now describe some basic properties of $\mathfrak{n}_{i,i+1}$ for later use.
\begin{lem}
For $i\in \mathbb{Z}_4$, the following relations hold in $\square_q$:
\begin{equation}
x_i\mathfrak{n}_{i,i+1}=q^{-2}\mathfrak{n}_{i,i+1}x_i,\qquad \qquad x_{i+1}\mathfrak{n}_{i,i+1}=q^{2}\mathfrak{n}_{i,i+1}x_{i+1}. \label{equ:uu}
\end{equation}
\end{lem}
\begin{proof}
To verify ({\ref {equ:uu}}), eliminate $\mathfrak{n}_{i,i+1}$ using the first equality in ({\ref {equ:3}}) and simplify the result using ({\ref {equ:1}}).
\end{proof}

\begin{cor}
For $i\in \mathbb{Z}_4$ and $n\in \mathbb{N}$, the following relations hold in $\square_q$:
\begin{equation}
x_i^n\mathfrak{n}_{i,i+1}=q^{-2n}\mathfrak{n}_{i,i+1}x_i^n,\qquad \qquad x_{i+1}^n\mathfrak{n}_{i,i+1}=q^{2n}\mathfrak{n}_{i,i+1}x_{i+1}^n. \label{equ:4}
\end{equation}
\end{cor}
\begin{proof}
By ({\ref {equ:uu}}) and induction on $n$.
\end{proof}

\begin{lem}
For $i\in \mathbb{Z}_4$ and $n\in \mathbb{N}$, the following relations hold in $\square_q$:
\begin{gather}
x_i^nx_{i+1}^n\Big(1-(q^{-2n}-q^{-2n-2})\mathfrak{n}_{i,i+1}\Big)=x_i^{n+1}x_{i+1}^{n+1},\label{equ:3101}\\
\Big(1-(q^{2n+2}-q^{2n})\mathfrak{n}_{i,i+1}\Big)x_{i+1}^nx_i^n=x_{i+1}^{n+1}x_i^{n+1}. \label{equ:3102}
\end{gather}
\end{lem}
\begin{proof}
In order to verify these equations, eliminate $\mathfrak{n}_{i,i+1}$ using the first equality in ({\ref {equ:3}}) and simplify the result using ({\ref {equ:c2}}).
\end{proof}

\section{Some isomorphisms and antiisomorphisms for $\square_q$}
In this section, we introduce some isomorphisms and antiisomorphisms for $\square_q$. By an {\it automorphism} of $\square_q$ we mean an $\mathbb{F}$-algebra isomorphism from $\square_q$ to $\square_q$.
\begin{lem}
\label{lem:m1}
There exists an automorphism $\rho$ of $\square_q$ that sends $x_i\mapsto x_{i+1}$ for $i\in \mathbb{Z}_4$.
\end{lem}
\begin{proof}
By Definition {\ref {def:box}}.
\end{proof}
\begin{lem}
The map $\rho$ from Lemma {\ref {lem:m1}} sends $\mathfrak{n}_{i,i+1}\mapsto \mathfrak{n}_{i+1,i+2}$ for $i\in \mathbb{Z}_4$.
\end{lem}
\begin{proof}
By the definition of $\rho$ and Definition {\ref {def:n}}.
\end{proof}

We recall the notion of antiisomorphism. Given $\mathbb{F}$-algebras $\mathcal{A},\mathcal{B}$, a map $\gamma : \mathcal{A} \to \mathcal{B}$ is called an {\it antiisomorphism} whenever $\gamma$ is an isomorphism of $\mathbb{F}$-vector spaces and $(ab)^\gamma=b^\gamma a^\gamma$ for all $a,b\in \mathcal{A}$. An antiisomorphism $\gamma : \mathcal{A} \to \mathcal{A}$ is called an {\it antiautomorphism} of $\mathcal{A}$

\begin{lem}
\label{lem:m3}
There exists an antiautomorphism $\phi$ of $\square_q$ that sends
\begin{equation*}
x_0\leftrightarrow x_{1}, \qquad \qquad x_{2}\leftrightarrow x_{3}.
\end{equation*}
Moreover there exists an antiautomorphism $\varphi$ of $\square_q$ that sends
\begin{equation*}
x_1\leftrightarrow x_{2}, \qquad \qquad x_{3}\leftrightarrow x_{0}.
\end{equation*}
\end{lem}
\begin{proof}
By Definition {\ref {def:box}}.
\end{proof}
\begin{lem}
\label{lll}
The map $\phi$ from Lemma {\ref {lem:m3}} sends
\begin{equation*}
\mathfrak{n}_{0,1}\mapsto \mathfrak{n}_{0,1},\qquad \mathfrak{n}_{2,3}\mapsto \mathfrak{n}_{2,3}, \qquad \mathfrak{n}_{1,2}\leftrightarrow \mathfrak{n}_{3,0}.
\end{equation*}
Moreover the map $\varphi$ from Lemma {\ref {lem:m3}} sends
\begin{equation*}
\mathfrak{n}_{1,2}\mapsto \mathfrak{n}_{1,2},\qquad \mathfrak{n}_{3,0}\mapsto \mathfrak{n}_{3,0}, \qquad \mathfrak{n}_{0,1}\leftrightarrow \mathfrak{n}_{2,3}.
\end{equation*}
\end{lem}
\begin{proof}
By the definitions of $\phi,\varphi$ and Definition {\ref {def:n}}.
\end{proof}

\begin{lem}
\label{lem:d4}
The maps $\rho$ from Lemma {\ref {lem:m1}} and $\phi,\varphi$ from Lemma {\ref {lem:m3}} satisfy the following relations:
\begin{gather}
\rho^4=\phi^2=\varphi^2=(\rho\phi)^2=(\rho\varphi)^2=1,\label{d1}\\
\rho\phi=\varphi\rho, \qquad \qquad \rho\varphi=\phi\rho. \label{d2}
\end{gather}
\end{lem}
\begin{proof}
By the definitions of $\rho,\phi,\varphi$.
\end{proof}

Recall that the dihedral group $D_4$ has the following group presentation:
\begin{equation*}
D_4=\{x,y\mid x^4=y^2=(xy)^2=1\}.
\end{equation*}
The group $D_4$ has $8$ elements and is the group of symmetries of a square. Consider the group $\rm{AAut}(\square_q)$ consisting of the automorphisms and antiautomorphisms of $\square_q$. The group operation is composition.
\begin{lem}
\label{lem:dd}
Let $G$ denote the subgroup of $\rm{AAut}(\square_q)$ generated by the maps $\rho$ from Lemma {\ref {lem:m1}} and $\phi,\varphi$ from Lemma {\ref {lem:m3}}. Then $G$ is isomorphic to $D_4$.
\end{lem}
\begin{proof}
By (\ref{d1}) there exists a group homomorphism $f: D_4\to G$ that sends $x\mapsto \rho$ and $y\mapsto \phi$. By (\ref{d2}) the element $\varphi$ is in the image of $f$. Therefore $f$ is surjective. By Lemma {\ref{lem:stu}}, the map $f$ is injective. By these comments $f$ is an isomorphism. The result follows.
\end{proof}

We now relate $\square_q$ and $\square_{q^{-1}}$.
\begin{lem}
\label{lem:m2}
There exists an antiisomorphism $\dagger: \square_q \to \square_{q^{-1}}$ that sends $x_i \mapsto x_i$ for $i\in \mathbb{Z}_4$. Moreover $\dagger^2=1$.
\end{lem}
\begin{proof}
By Definition {\ref {def:box}}.
\end{proof}

In Definition {\ref {def:n}} we discussed an element $\mathfrak{n}_{i,i+1}\in \square_q$. We retain the notation $\mathfrak{n}_{i,i+1}$ for the corresponding element in $\square_{q^{-1}}$.
\begin{lem}
The map $\dagger$ from Lemma {\ref{lem:m2}} sends $\mathfrak{n}_{i,i+1}\mapsto -\mathfrak{n}_{i,i+1}$ for $i\in \mathbb{Z}_4$.
\end{lem}
\begin{proof}
By the definition of $\dagger$ and Definition {\ref {def:n}}.
\end{proof}

\section{The element $x_i$ is invertible on finite-dimensional $\square_q$-modules}
Let $V$ denote a nonzero finite-dimensional $\square_q$-module. In this section, we show that for $i\in \mathbb{Z}_4$ the action of $x_i$ on $V$ is invertible.

We first show that the action of $x_i$ on $V$ is not nilpotent.
\begin{lem}
\label{lem:nonil}
Let $V$ denote a nonzero finite-dimensional $\square_q$-module. For $i\in \mathbb{Z}_4$, the action of $x_i$ on $V$ is not nilpotent.
\end{lem}
\begin{proof}
Assume that $x_i$ is nilpotent on $V$. Then there exists a minimal positive integer $n$ such that $x_i^n=0$ on $V$. By ({\ref {equ:1}}), we have $n\ne 1$. By ({\ref {equ:c1}}) and since $q$ is not a root of unity, we have $x_i^{n-1}=0$ on $V$. This contradicts the minimality of $n$. The result follows.
\end{proof}

We will use the following notation. Let $V$ denote a finite-dimensional vector space over $\mathbb{F}$ and let $A\in {\rm{End}}(V)$. For $\theta\in \mathbb{F}$ define
\begin{equation*}
{V}_A(\theta)=\{v\in V\mid \exists n\in \mathbb{N}, (A-\theta I)^nv=0\}.
\end{equation*}
Observe that $\theta$ is an eigenvalue of $A$ if and only if ${V}_A(\theta)\ne 0$, and in this case ${V}_A(\theta)$ is the corresponding generalized eigenspace. The sum $V=\bigoplus_{\theta\in \mathbb{F}}{V}_A(\theta)$ is direct.

\begin{pro}
\label{pro:fin}
Let $V$ denote a nonzero finite-dimensional $\square_q$-module. For $i\in \mathbb{Z}_4$ the action of $x_i$ on $V$ is invertible.
\end{pro}
\begin{proof}
To show $x_i$ is invertible on $V$, it suffices to show that $0$ is not an eigenvalue of $x_i$. Consider the subspace $W={V}_{x_i}(0)$. We first show that $W$ is $\square_q$-invariant. By construction, $W$ is $x_i$-invariant. Pick $v\in W$. By the definition of $W$, there exists $m\in\mathbb{N}$ such that $x_i^mv=0$. By (\ref {equ:c1}) with $n-1=m$, we have $x_i^{m+1}x_{i+1}v=0$. Therefore $x_{i+1}v\in W$. By (\ref {equ:c2}) with $n-1=m$, we have $x_i^{m+1}x_{i+3}v=0$. Therefore $x_{i+3}v\in W$. By (\ref {equ:c3}) with $n-2=m$, we have $x_i^{m+2}x_{i+2}v=0$. Therefore $x_{i+2}v\in W$. We have shown that $W$ is invariant under $x_j$ for $j\in \mathbb{Z}_4$. Therefore $W$ is $\square_q$-invariant. By construction $x_i$ is nilpotent on $W$. Therefore $W=0$ in view of Lemma {\ref {lem:nonil}}. The result follows.
\end{proof}

Motivated by Proposition {\ref {pro:fin}}, we make the following definition.

\begin{1def}
\label{def:ginv}
For $i\in \mathbb{Z}_4$, let $x_i^{-1}$ denote the operator that acts on every nonzero finite-dimension $\square_q$-module as the inverse of $x_i$.
\end{1def}

We now give some formulas involving the operators $x_i^{-1}$.
\begin{lem}
\label{lem:inv1}
For $i\in \mathbb{Z}_4$ the following relations hold on every nonzero finite-dimension $\square_q$-module:
\begin{gather}
qx_{i+1}x_i^{-1}-q^{-1}x_i^{-1}x_{i+1}=(q-q^{-1})x_i^{-2},\label{equ:inv11}\\
qx_{i+1}^{-1}x_i-q^{-1}x_ix_{i+1}^{-1}=(q-q^{-1})x_{i+1}^{-2}.\label{equ:inv12}
\end{gather}
\end{lem}
\begin{proof}
The equation (\ref {equ:inv11}) (resp. (\ref {equ:inv12})) follows by Definition {\ref {def:ginv}} and applying $x_i^{-1}$ (resp. $x_{i+1}^{-1}$) to both sides of (\ref {equ:1}).
\end{proof}

\begin{lem}
\label{lem:inv2}
For $i\in \mathbb{Z}_4$ the following relations hold on every nonzero finite-dimension $\square_q$-module:
\begin{align}
\frac{qx_{i}^{-2}x_{i+1}^{-1}+q^{-1}x_{i+1}^{-1}x_i^{-2}}{q+q^{-1}}=x_{i+1}^{-1}x_i^{-1}x_{i+1}x_i^{-1}x_{i+1}^{-1}, \label{equ:inv21}\\
\frac{qx_{i}^{-1}x_{i+1}^{-2}+q^{-1}x_{i+1}^{-2}x_i^{-1}}{q+q^{-1}}=x_{i}^{-1}x_{i+1}^{-1}x_{i}x_{i+1}^{-1}x_{i}^{-1}. \label{equ:inv22}
\end{align}
\end{lem}
\begin{proof}
We first show (\ref{equ:inv21}). In (\ref {equ:inv11}) multiply each term on the left by $x_{i+1}^{-1}$ and on the right by $x_{i}^{-1}x_{i+1}^{-1}$ to get
\begin{equation}
\label{equ:46}
qx_{i}^{-2}x_{i+1}^{-1}-q^{-1}x_{i+1}^{-1}x_i^{-1}x_{i+1}x_i^{-1}x_{i+1}^{-1}=(q-q^{-1})x_{i+1}^{-1}x_i^{-3}x_{i+1}^{-1}.
\end{equation}
Similarly in (\ref {equ:inv11}), multiply each term on the left by $x_{i+1}^{-1}x_i^{-1}$ and on the right by $x_{i+1}^{-1}$ to get
\begin{equation}
\label{equ:47}
qx_{i+1}^{-1}x_i^{-1}x_{i+1}x_i^{-1}x_{i+1}^{-1}-q^{-1}x_{i+1}^{-1}x_i^{-2}=(q-q^{-1})x_{i+1}^{-1}x_i^{-3}x_{i+1}^{-1}.
\end{equation}
Subtract (\ref {equ:46}) from (\ref {equ:47}) and solve for $x_{i+1}^{-1}x_i^{-1}x_{i+1}x_i^{-1}x_{i+1}^{-1}$ to get (\ref {equ:inv21}). To get (\ref{equ:inv22}), apply the map $\phi$ from Lemma {\ref {lem:m3}} to each side of (\ref {equ:inv21}).
\end{proof}

\section{The element $\mathfrak{n}_{i,i+1}$ is nilpotent on finite-dimensional $\square_q$-modules}
Let $V$ denote a finite-dimensional $\square_q$-module. In this section, we show that for $i\in \mathbb{Z}_4$ the action of $\mathfrak{n}_{i,i+1}$ on $V$ is nilpotent.

\begin{lem}
\label{lem:chain}
Let $V$ denote a finite-dimensional $\square_q$-module and let $\theta\in \mathbb{F}$. Then for $i\in \mathbb{Z}_4$, we have $\mathfrak{n}_{i,i+1}{V}_{x_i}(\theta)\subseteq {V}_{x_i}(q^{-2}\theta)$.
\end{lem}
\begin{proof}
Pick $v\in {V}_{x_i}(\theta)$. We show $\mathfrak{n}_{i,i+1}v\in {V}_{x_i}(q^{-2}\theta)$. By the definition of ${V}_{x_i}(\theta)$, there exists $n\in \mathbb{N}$ such that $(x_i-\theta I)^nv=0$. By this and the left equation in (\ref {equ:4}), we have $(x_i-q^{-2}\theta I)^n\mathfrak{n}_{i,i+1}v=0$. Therefore $\mathfrak{n}_{i,i+1}v\in {V}_{x_i}(q^{-2}\theta)$. The result follows.
\end{proof}

\begin{pro}
\label{pro:nnil}
Let $V$ denote a finite-dimensional $\square_q$-module. For $i\in \mathbb{Z}_4$ the action of $\mathfrak{n}_{i,i+1}$ on $V$ is nilpotent.
\end{pro}
\begin{proof}
Assume that $V$ is nonzero; otherwise the result is trivial. It suffices to show that for each eigenvalue $\theta$ of $x_i$, there exists a positive integer $m$ such that $\mathfrak{n}_{i,i+1}^m{V}_{x_i}(\theta)=0$. By Proposition {\ref {pro:fin}} the scalar $0$ is not an eigenvalue of $x_i$. Therefore $\theta\ne 0$. Since $V$ has finite positive dimension and $q$ is not a root of unity, there exists a positive integer $m$ such that $\theta q^{-2j}$ is an eigenvalue of $x_i$ for $0\le j\le m-1$, but $\theta q^{-2m}$ is not an eigenvalue of $x_i$. By this and Lemma {\ref {lem:chain}}, we have $\mathfrak{n}_{i,i+1}^m{V}_{x_i}(\theta)\subseteq {V}_{x_i}(\theta q^{-2m})=0$. Therefore $\mathfrak{n}_{i,i+1}^m{V}_{x_i}(\theta)=0$. The result follows.
\end{proof}

\section{The $q$-exponential function}
In this section we obtain some results involving the $q$-exponential function.
\begin{1def}
\cite[p.~204]{t1}
\label{def:qexp}
Let $V$ denote a vector space over $\mathbb{F}$ with finite positive dimension. Let $\psi\in {\rm{End}}(V)$ be nilpotent. Define
\begin{equation}
\label{equ:qexp}
{\rm {exp}}_q (\psi)=\sum_{n\in\mathbb{N}}\frac{q^{n \choose 2 }}{[n]_q^!}\psi^n.
\end{equation}
\end{1def}
The following result is well known and readily verified.
\begin{lem}
\rm
\cite[\rm p.~204]{t1}
\label{lem:inv}
Referring to Definition {\ref {def:qexp}}, the map ${\rm {exp}}_q(\psi)$ is invertible and its inverse is
\begin{equation*}
{\rm {exp}}_{q^{-1}}(-\psi)=\sum_{n\in\mathbb{N}}\frac{(-1)^nq^{- {n \choose 2 }}}{[n]_q^!}\psi^n.
\end{equation*}
\end{lem}
We mention an identity for later use.

\begin{lem}
Referring to Definition {\ref {def:qexp}},
\begin{equation}
{\rm {exp}}_q(q^2\psi)(1-(q^2-1)\psi)={\rm {exp}}_q(\psi). \label{equ:q1}
\end{equation}
\end{lem}
\begin{proof}
To verify (\ref {equ:q1}), express each side as a power series in $\psi$ using ({\ref {equ:qexp}}).
\end{proof}

Pick $i\in \mathbb{Z}_4$. By Proposition {\ref{pro:nnil}}, the action of $\mathfrak{n}_{i,i+1}$ on every nonzero finite-dimensional $\square_q$-module is nilpotent. We view ${\rm {exp}}_q(\mathfrak{n}_{i,i+1})$ as an operator that acts on every nonzero finite-dimensional $\square_q$-module. For $i,j\in \mathbb{Z}_4$, consider the following two expressions:
\begin{equation}
\label{equ:71}
{\rm {exp}}_q(\mathfrak{n}_{i,i+1})x_j{\rm {exp}}_q(\mathfrak{n}_{i,i+1})^{-1}, \qquad {\rm {exp}}_q(\mathfrak{n}_{i,i+1})^{-1}x_j{\rm {exp}}_q(\mathfrak{n}_{i,i+1}).
\end{equation}
For each expression in (\ref {equ:71}), expand both $q$-exponential terms using Definition \ref{def:qexp} and Lemma \ref{lem:inv}. This yields a double sum with infinitely many terms. We will show that in fact, each double sum is a polynomial in $\{x_k^{\pm 1}\}_{k\in \mathbb{Z}_4}$. We now give some formulas for later use.
\begin{lem}
\label{lem:qexp}
For $i\in \mathbb{Z}_4$ and $r\in \mathbb{Z}$, the following relations hold on every nonzero finite-dimension $\square_q$-module:
\begin{gather}
{\rm {exp}}_q(q^{2r}\mathfrak{n}_{i,i+1})={\rm {exp}}_q(\mathfrak{n}_{i,i+1})x_i^{-r}x_{i+1}^{-r}, \label{equ:5511}\\ {\rm {exp}}_q(q^{2r}\mathfrak{n}_{i,i+1})=x_i^{-r}x_{i+1}^{-r}{\rm {exp}}_q(\mathfrak{n}_{i,i+1}). \label{equ:5512}
\end{gather}
\end{lem}
\begin{proof}
To show ({\ref {equ:5511}}) for $r\geq 0$, use induction on $r$. The calculation is routine using ({\ref {equ:q1}}) with $\psi=q^{2r}\mathfrak{n}_{i,i+1}$ along with ({\ref {equ:3102}}). We similarly show ({\ref {equ:5511}}) for $r< 0$ by induction on $r=-1,-2,\dots$ using ({\ref {equ:3101}}) and ({\ref {equ:q1}}). To get ({\ref {equ:5512}}), apply the map $\phi$ from Lemma {\ref {lem:m3}} to each side of ({\ref {equ:5511}}) and use Lemma {\ref {lll}}.
\end{proof}

\section{Some $q$-exponential formulas, part I}
In this section, we analyze ({\ref {equ:71}}) for the case $j=i$ or $j=i+1$. The following Theorem {\ref {pro:81}} is a variation of \cite[Lemma~5.8, 5.9]{equit}.

\begin{thm}
\label{pro:81}
For $i\in \mathbb{Z}_4$, the following relations hold on every nonzero finite-dimension $\square_q$-module:
\begin{gather}
{\rm {exp}}_q(\mathfrak{n}_{i,i+1})x_{i}{\rm {exp}}_q(\mathfrak{n}_{i,i+1})^{-1}=x_{i+1}^{-1}, \label{equ:vip2}\\
{\rm {exp}}_q(\mathfrak{n}_{i,i+1})^{-1}x_{i+1}{\rm {exp}}_q(\mathfrak{n}_{i,i+1})=x_i^{-1}.\label{equ:vip1}
\end{gather}
\end{thm}
\begin{proof}
We first verify (\ref {equ:vip1}). By the equation on the right in (\ref {equ:4}) and Definition {\ref {def:qexp}}, we have
\begin{equation*}
x_{i+1}{\rm {exp}}_q(\mathfrak{n}_{i,i+1})x_{i+1}^{-1}={\rm {exp}}_q(q^2\mathfrak{n}_{i,i+1}).
\end{equation*}
Using this and (\ref {equ:5511}) with $r=1$ we routinely obtain (\ref {equ:vip1}). To get (\ref {equ:vip2}), apply the map $\phi$ from Lemma {\ref {lem:m3}} to each side of (\ref {equ:vip1}).
\end{proof}

The following Theorem {\ref {pro:82}} is a variation of \cite[Lemma~6.1, 6.2]{equit}.
\begin{thm}
\label{pro:82}
For $i\in \mathbb{Z}_4$, the following relations hold on every nonzero finite-dimension $\square_q$-module:
\begin{gather}
{\rm {exp}}_q(\mathfrak{n}_{i,i+1})^{-1}x_{i}{\rm {exp}}_q(\mathfrak{n}_{i,i+1})=x_ix_{i+1}x_i,\label{equ:vip3}\\
{\rm {exp}}_q(\mathfrak{n}_{i,i+1})x_{i+1}{\rm {exp}}_q(\mathfrak{n}_{i,i+1})^{-1}=x_{i+1}x_ix_{i+1}. \label{equ:vip4}
\end{gather}
\end{thm}
\begin{proof}
We first verify (\ref {equ:vip3}). By (\ref {equ:4}) the element $x_ix_{i+1}$ commutes with $\mathfrak{n}_{i,i+1}$. Therefore ${\rm {exp}}_q(\mathfrak{n}_{i,i+1})^{-1}x_ix_{i+1}{\rm {exp}}_q(\mathfrak{n}_{i,i+1})=x_ix_{i+1}$ in view of Definition {\ref {def:qexp}}. Combine this equation with (\ref {equ:vip1}) to get (\ref {equ:vip3}). To get (\ref {equ:vip4}), apply the map $\phi$ from Lemma {\ref {lem:m3}} to each side of (\ref {equ:vip3}).
\end{proof}

\section{Some $q$-exponential formulas, part II}
In this section, we analyze ({\ref {equ:71}}) for the case $j=i+2$ or $j=i+3$.
\begin{lem}
\label{lem:rec1}
For $i\in \mathbb{Z}_4$, the following relations hold in ${\square}_q$:
\begin{gather}
\sum_{m=0}^{3}(-1)^mq^{3-2m}\frac{\mathfrak{n}_{i,i+1}^{3-m}}{[3-m]_q^!}{x}_{i+2}\frac{\mathfrak{n}_{i,i+1}^m}{[m]_q^!}=-(q-q^{-1})^2\mathfrak{n}_{i,i+1}{x}_i\mathfrak{n}_{i,i+1}, \label{equ:rec1}\\
\sum_{m=0}^{3}(-1)^mq^{3-2m}\frac{\mathfrak{n}_{i,i+1}^{m}}{[m]_q^!}{x}_{i+3}\frac{\mathfrak{n}_{i,i+1}^{3-m}}{[3-m]_q^!}=-(q-q^{-1})^2\mathfrak{n}_{i,i+1}{x}_{i+1}\mathfrak{n}_{i,i+1}. \label{equ:rec2}
\end{gather}
\end{lem}
\begin{proof}
To verify (\ref {equ:rec1}) let $\Theta$ denote the left-hand side of (\ref {equ:rec1}) minus the right-hand side of (\ref {equ:rec1}). We will show that $\Theta=0$. To do this we first eliminate each occurrence of $\mathfrak{n}_{i,i+1}$ in $\Theta$ using the second equality in ({\ref {equ:3}}). In the resulting equation, we simplify things using the following principle: for each occurrence of $x_{i+1}$, move it to the leftmost factor using (\ref {equ:1}). The above simplification yields the following results.\\

The expression $q^3(q-q^{-1})^3\mathfrak{n}_{i,i+1}^3{x}_{i+2}$ is a weighted sum involving the following terms and coefficients:\\
\\
\bigskip
\centerline{
\begin{tabular}[t]{c|c c c c}
 term & ${x}_{i+2}$ & ${x}_{i+1}{x}_i{x}_{i+2}$ & ${x}_{i+1}^2{x}_i^2{x}_{i+2}$& ${x}_{i+1}^3{x}_i^3{x}_{i+2}$
   \\ \hline
coeff. & $1$ &
$-q^{-2}[3]_q$ & $q^{-4}[3]_q$ & $-q^{-6}$
\\
 \end{tabular}}

The expression $q^3(q-q^{-1})^3\mathfrak{n}_{i,i+1}^2{x}_{i+2}\mathfrak{n}_{i,i+1}$ is a weighted sum involving the following terms and coefficients:\\
\\
\bigskip
\centerline{
\begin{tabular}[t]{c|c c c c c}
 term & ${x}_{i+2}$ & ${x}_{i}$ &  ${x}_{i+1}{x}_i{x}_{i+2}$ & ${x}_{i+1}{x}_{i+2}{x}_{i}$& ${x}_{i+1}{x}_i^2$
   \\ \hline
coeff. & $1$ & $q^2-1$ &
$-1-q^{-2}$ & $-q^{-2}$ & $q^{-2}-q^{2}$
\\
 \end{tabular}}
     \bigskip
\centerline{
\begin{tabular}[t]{c|c c c c}
 term & ${x}_{i+1}^2{x}_{i}^2{x}_{i+2}$ & ${x}_{i+1}^2{x}_i{x}_{i+2}{x}_{i}$ & ${x}_{i+1}^2{x}_i^3$& ${x}_{i+1}^3{x}_i^2{x}_{i+2}{x}_{i}$
   \\ \hline
coeff. & $q^{-2}$ &
$q^{-2}+q^{-4}$ & $1-q^{-2}$ & $-q^{-4}$
\\
 \end{tabular}}

The expression $q^3(q-q^{-1})^3\mathfrak{n}_{i,i+1}{x}_{i+2}\mathfrak{n}_{i,i+1}^2$ is a weighted sum involving the following terms and coefficients:\\
\\
\bigskip
\centerline{
\begin{tabular}[t]{c|c c c c c}
 term & ${x}_{i+2}$ & ${x}_{i}$ &  ${x}_{i+1}{x}_i{x}_{i+2}$ & ${x}_{i+1}{x}_{i+2}{x}_{i}$& ${x}_{i+1}{x}_i^2$
   \\ \hline
coeff. & $1$ & $q^2-q^{-2}$ &
$-1$ & $-1-q^{-2}$ & $q^{-1}(q^{-1}+q)(q^{-2}-q^2)$
\\
 \end{tabular}}
     \bigskip
\centerline{
\begin{tabular}[t]{c|c c c c}
 term & ${x}_{i+1}^2{x}_{i+2}{x}_{i}^2$ & ${x}_{i+1}^2{x}_i{x}_{i+2}{x}_{i}$ & ${x}_{i+1}^2{x}_i^3$& ${x}_{i+1}^3{x}_i{x}_{i+2}{x}_{i}^2$
   \\ \hline
coeff. & $q^{-2}$ &
$q^{-2}+1$ & $1-q^{-4}$ & $-q^{-2}$
\\
 \end{tabular}}

The expression $q^3(q-q^{-1})^3{x}_{i+2}\mathfrak{n}_{i,i+1}^3$ is a weighted sum involving the following terms and coefficients:\\
\\
\bigskip
\centerline{
\begin{tabular}[t]{c|c c c c}
 term & ${x}_{i+2}$ & ${x}_{i}$ &  ${x}_{i+1}{x}_{i+2}{x}_{i}$ &  ${x}_{i+1}{x}_i^2$
   \\ \hline
coeff. & $1$ & $q^2-q^{-4}$ &
$-[3]_q$ & $q^{-2}(q^{-2}-q^2)[3]_q$
\\
 \end{tabular}}
     \bigskip
\centerline{
\begin{tabular}[t]{c|c c c}
 term & ${x}_{i+1}^2{x}_{i+2}{x}_{i}^2$  & ${x}_{i+1}^2{x}_i^3$& ${x}_{i+1}^3{x}_{i+2}{x}_{i}^3$
   \\ \hline
coeff. & $[3]_q$ &
$1-q^{-6}$ & $-1$
\\
 \end{tabular}}

The expression $q^4(q-q^{-1})^2\mathfrak{n}_{i,i+1}{x}_i\mathfrak{n}_{i,i+1}$ is a weighted sum involving the following terms and coefficients:\\
\\
\bigskip
\centerline{
\begin{tabular}[t]{c|c c c}
 term & ${x}_i$ & ${x}_{i+1}{x}_i^2$ & ${x}_{i+1}^2{x}_i^3$
   \\ \hline
coeff. & $1$ &
$-1-q^{-2}$ & $q^{-2}$
\\
 \end{tabular}}

By the above comments $\Theta$ is equal to
\begin{equation}
\label{equ:tem}
\frac{-{x}_{i+1}^3({x}_i^3{x}_{i+2}-[3]_q{x}_i^2{x}_{i+2}{x}_i+[3]_q{x}_i{x}_{i+2}{x}_i^2-{x}_{i+2}{x}_i^3)}{q^6(q-q^{-1})^3[3]_q^!}.
\end{equation}

The expression (\ref {equ:tem}) is $0$ by (\ref {equ:2}). Therefore $\Theta=0$. We have shown (\ref {equ:rec1}). To get (\ref {equ:rec2}) apply the map ${\phi}$ from Lemma {\ref {lem:m3}} to each side of (\ref {equ:rec1}).
\end{proof}

\begin{lem}
\label{lem:rec3}
For $i\in \mathbb{Z}_4$ and $m\in\mathbb{N}$, the following relations hold in ${\square}_q$:
\begin{equation}
\begin{aligned}
\label{equ:rac1}
x_{i+2}\mathfrak{n}_{i,i+1}^m&=
a_m(q)\mathfrak{n}_{i,i+1}^mx_{i+2}+b_m(q)\mathfrak{n}_{i,i+1}^{m-1}x_{i+2}\mathfrak{n}_{i,i+1}\\
&+\ c_m(q)\mathfrak{n}_{i,i+1}^{m-2}x_{i+2}n_{i,i+1}^2+d_m(q)\mathfrak{n}_{i,i+1}^{m-1}x_{i},
\end{aligned}
\end{equation}
\begin{equation}
\begin{aligned}
\label{equ:rac2}
\mathfrak{n}_{i,i+1}^mx_{i+3}&=a_m(q)x_{i+3}\mathfrak{n}_{i,i+1}^m+b_m(q)\mathfrak{n}_{i,i+1}x_{i+3}\mathfrak{n}_{i,i+1}^{m-1}\\
&+\ c_m(q)\mathfrak{n}_{i,i+1}^2x_{i+3}\mathfrak{n}_{i,i+1}^{m-2}+d_m(q)x_{i+1}\mathfrak{n}_{i,i+1}^{m-1},
\end{aligned}
\end{equation}
\begin{equation}
\begin{aligned}
\label{equ:rac3}
\mathfrak{n}_{i,i+1}^mx_{i+2}&=a_m(q^{-1})x_{i+2}\mathfrak{n}_{i,i+1}^m+b_m(q^{-1})\mathfrak{n}_{i,i+1}x_{i+2}\mathfrak{n}_{i,i+1}^{m-1} \\
&+\ c_m(q^{-1})\mathfrak{n}_{i,i+1}^2x_{i+2}\mathfrak{n}_{i,i+1}^{m-2}-d_m(q^{-1})x_{i}\mathfrak{n}_{i,i+1}^{m-1},
\end{aligned}
\end{equation}
\begin{equation}
\begin{aligned}
\label{equ:rac4}
x_{i+3}\mathfrak{n}_{i,i+1}^m&=a_m(q^{-1})\mathfrak{n}_{i,i+1}^mx_{i+3}+b_m(q^{-1})\mathfrak{n}_{i,i+1}^{m-1}x_{i+3}\mathfrak{n}_{i,i+1} \\
&+\ c_m(q^{-1})\mathfrak{n}_{i,i+1}^{m-2}x_{i+3}\mathfrak{n}_{i,i+1}^2-d_m(q^{-1})\mathfrak{n}_{i,i+1}^{m-1}x_{i+1},
\end{aligned}
\end{equation}
where
\begin{gather*}
a_m(q)=q^{2m}\frac{[m-1]_q[m-2]_q}{[2]_q},\\
b_m(q)=-q^{2m-2}[m]_q[m-2]_q,\\
c_m(q)=q^{2m-4}\frac{[m]_q[m-1]_q}{[2]_q},\\
d_m(q)=q^{m-5}[3m]_{q}-q^{-3}[3]_q[2m]_q+q^{-m-1}[3]_q[m]_q.
\end{gather*}
\end{lem}
\begin{proof}
To get (\ref {equ:rac1}), use (\ref {equ:2}), (\ref{equ:rec1}) and induction on $m$. To get (\ref {equ:rac2}), apply the map ${\phi}$ from Lemma {\ref {lem:m3}} to each side of (\ref {equ:rac1}). Concerning (\ref {equ:rac3}), apply the map $\dagger$ from Lemma {\ref {lem:m2}} to each side of (\ref {equ:rac1}). This yields an equation that holds in ${\square}_{q^{-1}}$. In this equation replace $q$ by $q^{-1}$. This gives (\ref {equ:rac3}). To get (\ref {equ:rac4}), apply ${\phi}$ to each side of (\ref {equ:rac3}).
\end{proof}

We now analyze ({\ref {equ:71}}) for the case $j=i+3$.
\begin{thm}
\label{pro:vip1}
For $i\in \mathbb{Z}_4$, the following relation holds on every nonzero finite-dimensional ${\square}_q$-module:
\begin{flalign*}
{\rm {exp}}_q(\mathfrak{n}_{i,i+1})^{-1}x_{i+3}{\rm {exp}}_q(\mathfrak{n}_{i,i+1}) =
x_{i+3}-x_{i}^{-1}+\frac{qx_{i}x_{i+1}x_{i+3}}{q-q^{-1}}
-\frac{x_ix_{i+3}x_{i+1}}{q(q-q^{-1})}\\
+\frac{q^3x_i^2x_{i+1}^2x_{i+3}}{(q-q^{-1})(q^2-q^{-2})}
+\frac{qx_i^2x_{i+3}x_{i+1}^2}{(q-q^{-1})(q^2-q^{-2})}
-\frac{q^2x_i^2x_{i+1}x_{i+3}x_{i+1}}{(q-q^{-1})^2}.
\end{flalign*}
\end{thm}
\begin{proof}
For $m\in\mathbb{N}$ multiply each side of (\ref{equ:rac4}) by ${q^{m \choose 2 }}/{[m]_q^!}$. Sum the resulting equations over $m\in\mathbb{N}$ and evaluate the result using (\ref{equ:qexp}) to get
\begin{equation*}
\begin{aligned}
x_{i+3}{\rm {exp}}_q(\mathfrak{n}_{i,i+1})&=\frac{q^3{\rm {exp}}_q(q^{-4}\mathfrak{n}_{i,i+1})x_{i+3}}{(q-q^{-1})(q^2-q^{-2})}
-\frac{{\rm {exp}}_q(q^{-2}\mathfrak{n}_{i,i+1})x_{i+3}}{(q-q^{-1})^2}\\
&+\frac{{\rm {exp}}_q(\mathfrak{n}_{i,i+1})x_{i+3}}{q^3(q-q^{-1})(q^2-q^{-2})}
-\frac{{\rm {exp}}_q(\mathfrak{n}_{i,i+1})x_{i+3}\mathfrak{n}_{i,i+1}}{q(q-q^{-1})}\\
&+\frac{q{\rm {exp}}_q(q^{-2}\mathfrak{n}_{i,i+1})x_{i+3}\mathfrak{n}_{i,i+1}}{q-q^{-1}}
+\frac{q{\rm {exp}}_q(\mathfrak{n}_{i,i+1})x_{i+3}\mathfrak{n}_{i,i+1}^2}{q+q^{-1}}\\
&-{\rm {exp}}_q(q^2\mathfrak{n}_{i,i+1})x_{i+1}
+(1+q^2){\rm {exp}}_q(\mathfrak{n}_{i,i+1})x_{i+1}\\
&-q^2{\rm {exp}}_q(q^{-2}\mathfrak{n}_{i,i+1})x_{i+1}.
\end{aligned}
\end{equation*}

In the above equation multiply each term on the left by ${\rm {exp}}_q(\mathfrak{n}_{i,i+1})^{-1}$ and use (\ref {equ:5511}) to get that ${\rm {exp}}_q(\mathfrak{n}_{i,i+1})^{-1}x_{i+3}{\rm {exp}}_q(\mathfrak{n}_{i,i+1})$ is equal to
\begin{equation*}
\begin{aligned}
&\frac{q^3x_i^2x_{i+1}^2x_{i+3}}{(q-q^{-1})(q^2-q^{-2})}
-\frac{x_ix_{i+1}x_{i+3}}{(q-q^{-1})^2}
+\frac{x_{i+3}}{q^3(q-q^{-1})(q^2-q^{-2})}\\
&-\frac{x_{i+3}\mathfrak{n}_{i,i+1}}{q(q-q^{-1})}
+\frac{qx_{i}x_{i+1}x_{i+3}\mathfrak{n}_{i,i+1}}{q-q^{-1}}
+\frac{qx_{i+3}\mathfrak{n}_{i,i+1}^2}{q+q^{-1}}\\
&-x_i^{-1}+(1+q^2)x_{i+1}-q^2x_ix_{i+1}^2.
\end{aligned}
\end{equation*}

For notational convenience let $\Psi$ denote the above expression. In $\Psi$ we first eliminate every occurrence of $\mathfrak{n}_{i,i+1}$ using the second equality in ({\ref {equ:3}}). In the resulting expression, we simplify things using the following principle: for each occurrence of $x_{i}$, move it to the far left using (\ref {equ:1}). The above simplification yields the following results.

The expression $-{q^{-1}(q-q^{-1})^{-1}}{x_{i+3}\mathfrak{n}_{i,i+1}}$ is a weighted sum involving the following terms and coefficients:\\
\\
\bigskip
\centerline{
\begin{tabular}[t]{c|c c c}
 term & ${x}_{i+3}$ & ${x}_{i}{x}_{i+3}{x}_{i+1}$ & ${x}_{i+1}$
   \\ \hline
coeff. & $-{(q-q^{-1})^{-2}}$ &
$q^{-2}(q-q^{-1})^{-2}$ & ${q^{-1}(q-q^{-1})^{-1}}$
\\
 \end{tabular}}

The expression $q{(q-q^{-1})^{-1}}x_{i}x_{i+1}x_{i+3}\mathfrak{n}_{i,i+1}$ is a weighted sum involving the following terms and coefficients:\\
\\
\bigskip
\centerline{
\begin{tabular}[t]{c|c c c c}
 term & $x_i{x}_{i+1}x_{i+3}$ & ${x}_{i}^2{x}_{i+1}{x}_{i+3}x_{i+1}$ & ${x}_{i}{x}_{i+3}{x}_{i+1}$& ${x}_{i}{x}_{i+1}^2$
   \\ \hline
coeff. & ${q^2}{(q-q^{-1})^{-2}}$ &
$-{q^2}{(q-q^{-1})^{-2}}$ & ${q}{(q-q^{-1})^{-1}}$ & $-{q}{(q-q^{-1})^{-1}}$
\\
 \end{tabular}}

The expression ${q(q+q^{-1})^{-1}x_{i+3}\mathfrak{n}_{i,i+1}^2}$ is a weighted sum involving the following terms and coefficients:\\
\\
\bigskip
\centerline{
\begin{tabular}[t]{c|c c c}
 term & $x_{i+3}$ & ${x}_{i}{x}_{i+3}x_{i+1}$ & ${x}_{i+1}$
   \\ \hline
coeff. & ${q^3}{(q-q^{-1})^{-1}(q^2-q^{-2})^{-1}}$ &
$-{q^2}{(q-q^{-1})^{-2}}$ & $-{q^3}{(q-q^{-1})^{-1}}$
\\
 \end{tabular}}
\bigskip
\centerline{
\begin{tabular}[t]{c| c c}
 term &  ${x}_{i}^2x_{i+3}{x}_{i+1}^2$ & $x_ix_{i+1}^2$
   \\ \hline
coeff.  & ${q}{(q-q^{-1})^{-1}(q^2-q^{-2})^{-1}}$ & ${q^3}{(q-q^{-1})^{-1}}$
\\
 \end{tabular}}

Evaluating $\Psi$ using the above comments, we get the result.
\end{proof}

\begin{thm}
\label{pro:vip2}
For $i\in \mathbb{Z}_4$, the following relation holds on every nonzero finite-dimensional ${\square}_q$-module:
\begin{flalign*}
{\rm {exp}}_q(\mathfrak{n}_{i,i+1})x_{i+3}{\rm {exp}}_q(\mathfrak{n}_{i,i+1})^{-1}&=
x_{i+1}-\frac{x_{i+3}}{(q-q^{-1})^2}+\frac{qx_{i+1}^{-1}x_{i+3}x_{i+1}}{(q-q^{-1})(q^2-q^{-2})}\\
&+\frac{q^{-1}x_{i+1}x_{i+3}x_{i+1}^{-1}}{(q-q^{-1})(q^2-q^{-2})}.
\end{flalign*}
\end{thm}
\begin{proof}
For $m\in\mathbb{N}$ multiply each side of (\ref {equ:rac2}) by $q^{-2m}{q^{m \choose 2 }}/{[m]_q^!}$. Sum the resulting equations over $m\in\mathbb{N}$ and evaluate the result using (\ref{equ:qexp}) to get
\begin{equation*}
\begin{aligned}
{\rm {exp}}_q(q^{-2}\mathfrak{n}_{i,i+1})x_{i+3}&=\frac{x_{i+3}{\rm {exp}}_q(q^2\mathfrak{n}_{i,i+1})}{q^3(q-q^{-1})(q^2-q^{-2})}
-\frac{x_{i+3}{\rm {exp}}_q(\mathfrak{n}_{i,i+1})}{(q-q^{-1})^2}\\
&+\frac{q^3x_{i+3}{\rm {exp}}_q(q^{-2}\mathfrak{n}_{i,i+1})}{(q-q^{-1})(q^2-q^{-2})}
+\frac{\mathfrak{n}_{i,i+1}x_{i+3}{\rm {exp}}_q(\mathfrak{n}_{i,i+1})}{q(q-q^{-1})}\\
&-\frac{\mathfrak{n}_{i,i+1}x_{i+3}{\rm {exp}}_q(q^2\mathfrak{n}_{i,i+1})}{q^3(q-q^{-1})}
+\frac{\mathfrak{n}_{i,i+1}^2x_{i+3}{\rm {exp}}_q(q^2\mathfrak{n}_{i,i+1})}{q^3(q+q^{-1})}\\
&+\frac{x_{i+1}{\rm {exp}}_q(q^{-2}\mathfrak{n}_{i,i+1})}{q^2}
-\frac{(q+q^{-1})x_{i,i+1}{\rm {exp}}_q(\mathfrak{n}_{i,i+1})}{q^3}\\
&+\frac{x_{i+1}{\rm {exp}}_q(q^2\mathfrak{n}_{i,i+1})}{q^4}.
\end{aligned}
\end{equation*}

In the above equation multiply each term on the left by $x_{i+1}^{-1}x_i^{-1}$ and on the right by ${\rm {exp}}_q(\mathfrak{n}_{i,i+1})^{-1}$, and then use (\ref {equ:5512}) to get that ${\rm {exp}}_q(\mathfrak{n}_{i,i+1})x_{i+3}{\rm {exp}}_q(\mathfrak{n}_{i,i+1})^{-1}$ is equal to
\begin{equation*}
\begin{aligned}
&\frac{x_{i+1}^{-1}x_i^{-1}x_{i+3}x_i^{-1}x_{i+1}^{-1}}{q^3(q-q^{-1})(q^2-q^{-2})}
-\frac{x_{i+1}^{-1}x_i^{-1}x_{i+3}}{(q-q^{-1})^2}
+\frac{q^3x_{i+1}^{-1}x_i^{-1}x_{i+3}x_{i}x_{i+1}}{(q-q^{-1})(q^2-q^{-2})}\\
&+\frac{x_{i+1}^{-1}x_i^{-1}\mathfrak{n}_{i,i+1}x_{i+3}}{q(q-q^{-1})}
-\frac{x_{i+1}^{-1}x_i^{-1}\mathfrak{n}_{i,i+1}x_{i+3}x_{i}^{-1}x_{i+1}^{-1}}{q^3(q-q^{-1})}\\
&+\frac{x_{i+1}^{-1}x_i^{-1}\mathfrak{n}_{i,i+1}^2x_{i+3}x_{i}^{-1}x_{i+1}^{-1}}{q^3(q+q^{-1})}
+\frac{x_{i+1}^{-1}x_i^{-1}x_{i+1}x_{i}x_{i+1}}{q^2}\\
&-\frac{(q+q^{-1})x_{i+1}^{-1}x_i^{-1}x_{i+1}}{q^3}+\frac{x_{i+1}^{-1}x_i^{-1}x_{i+1}x_{i}^{-1}x_{i+1}^{-1}}{q^4}.
\end{aligned}
\end{equation*}

For notational convenience let $\Phi$ denote the above expression. In $\Phi$ we first eliminate every occurrence of $\mathfrak{n}_{i,i+1}$ using the first equality in ({\ref {equ:3}}). In the resulting expression, we simplify things using (\ref{equ:1}), (\ref {equ:inv11}), and (\ref {equ:inv12}). Our guiding principle is to bring $x_i,x_i^{-1}$ together for cancellation, and also to bring $x_{i+1},x_{i+1}^{-1}$ together for cancellation. The above simplification yields the following results.

The expression ${q^3{(q-q^{-1})^{-1}(q^2-q^{-2})^{-1}}x_{i+1}^{-1}x_i^{-1}x_{i+3}x_{i}x_{i+1}}$ is a weighted sum involving the following terms and coefficients:\\
\\
\bigskip
\centerline{
\begin{tabular}[t]{c|c c c}
 term & ${x}_{i+1}^{-1}x_{i+3}x_{i+1}$ & ${x}_{i}^{-1}$ & ${x}_{i+1}^{-1}x_i^{-2}$
   \\ \hline
coeff. & ${q}(q-q^{-1})^{-1}(q^2-q^{-2})^{-1}$ &
${q^4}{(q^2-q^{-2})^{-1}}$ & $-{q^3}{(q+q^{-1})^{-1}}$
\\
 \end{tabular}}

The expression $q^{-1}(q-q^{-1})^{-1}{x_{i+1}^{-1}x_i^{-1}\mathfrak{n}_{i,i+1}x_{i+3}}$ is a weighted sum involving the following terms and coefficients:\\
\\
\bigskip
\centerline{
\begin{tabular}[t]{c|c c}
 term & ${x}_{i+1}^{-1}x_i^{-1}x_{i+3}$ & ${x}_{i+3}$
   \\ \hline
coeff. & ${(q-q^{-1})^{-2}}$ &
$-(q-q^{-1})^{-2}$
\\
 \end{tabular}}

The expression $-q^{-3}(q-q^{-1})^{-1}{x_{i+1}^{-1}x_i^{-1}\mathfrak{n}_{i,i+1}x_{i+3}x_{i}^{-1}x_{i+1}^{-1}}$ is a weighted sum involving the following terms and coefficients:\\
\\
\bigskip
\centerline{
\begin{tabular}[t]{c|c c}
 term & ${x}_{i+1}^{-1}x_i^{-1}x_{i+3}x_i^{-1}{x}_{i+1}^{-1}$ & ${x}_{i+3}x_i^{-1}{x}_{i+1}^{-1}$
   \\ \hline
coeff. & $-{q^{-2}(q-q^{-1})^{-2}}$ &
${q^{-2}(q-q^{-1})^{-2}}$
\\
 \end{tabular}}

The expression $q^{-3}(q+q^{-1})^{-1}{x_{i+1}^{-1}x_i^{-1}\mathfrak{n}_{i,i+1}^2x_{i+3}x_{i}^{-1}x_{i+1}^{-1}}$ is a weighted sum involving the following terms and coefficients:\\
\\
\bigskip
\centerline{
\begin{tabular}[t]{c|c c }
 term & ${x}_{i+1}^{-1}x_i^{-1}x_{i+3}x_i^{-1}{x}_{i+1}^{-1}$ & ${x}_{i+3}x_i^{-1}{x}_{i+1}^{-1}$
   \\ \hline
coeff. & ${q^{-1}(q-q^{-1})^{-1}(q^2-q^{-2})^{-1}}$ &
$-{q^{-2}(q-q^{-1})^{-2}}$
\\
 \end{tabular}}
\bigskip
\centerline{
\begin{tabular}[t]{c|c c c}
 term & ${x}_{i}^{-1}$ & $x_i^{-2}x_{i+1}^{-1}$ & ${x}_{i+1}x_{i+3}{x}_{i+1}^{-1}$
   \\ \hline
coeff. & $-{q^{-4}(q^2-q^{-2})^{-1}}$ & $-{q^{-3}(q+q^{-1})^{-1}}$ & $q^{-1}(q-q^{-1})^{-1}(q^2-q^{-2})^{-1}$
\\
 \end{tabular}}

The expression $q^{-2}{x_{i+1}^{-1}x_i^{-1}x_{i+1}x_{i}x_{i+1}}$ is a weighted sum involving the following terms and coefficients:\\
\\
\bigskip
\centerline{
\begin{tabular}[t]{c|c c c}
 term & ${x}_{i+1}$ & ${x}_{i}^{-1}$ & ${x}_{i+1}^{-1}x_i^{-2}$
   \\ \hline
coeff. & $1$ &
$1-q^2$ & $(q-q^{-1})^2$
\\
 \end{tabular}}

The expression $-{q^{-3}(q+q^{-1})x_{i+1}^{-1}x_i^{-1}x_{i+1}}$ is a weighted sum involving the following terms and coefficients:\\
\\
\bigskip
\centerline{
\begin{tabular}[t]{c|c c}
 term & ${x}_{i}^{-1}$ & $x_{i+1}^{-1}{x}_{i}^{-2}$
   \\ \hline
coeff. & $-q^{-1}(q+q^{-1})$ &
$q^{-2}{(q^2-q^{-2})}$
\\
 \end{tabular}}

Using (\ref {equ:inv21}) the expression $q^{-4}x_{i+1}^{-1}x_i^{-1}x_{i+1}x_{i}^{-1}x_{i+1}^{-1}$ is a weighted sum involving the following terms and coefficients:\\
\\
\bigskip
\centerline{
\begin{tabular}[t]{c|c c}
 term & ${x}_{i}^{-2}x_{i+1}^{-1}$ & $x_{i+1}^{-1}{x}_{i}^{-2}$
   \\ \hline
coeff. & $q^{-3}(q+q^{-1})^{-1}$ &
$q^{-5}(q+q^{-1})^{-1}$
\\
 \end{tabular}}

Evaluating $\Phi$ using the above comments we get the result.
\end{proof}

We now analyze ({\ref {equ:71}}) for the case $j=i+2$.
\begin{thm}
\label{pro:vip3}
For $i\in \mathbb{Z}_4$, the following relation holds on every nonzero finite-dimensional ${\square}_q$-module:
\begin{align*}
{\rm {exp}}_q(\mathfrak{n}_{i,i+1})^{-1}x_{i+2}{\rm {exp}}_q(\mathfrak{n}_{i,i+1})&=
x_{i}-\frac{x_{i+2}}{(q-q^{-1})^2}
+\frac{qx_{i}x_{i+2}x_{i}^{-1}}{(q-q^{-1})(q^2-q^{-2})}\\
&+\frac{q^{-1}x_{i}^{-1}x_{i+2}x_{i}}{(q-q^{-1})(q^2-q^{-2})}.
\end{align*}
\end{thm}
\begin{proof}
Apply the map ${\phi}$ from Lemma {\ref {lem:m3}} to each side of the equation in Theorem {\ref {pro:vip2}}.
\end{proof}

\begin{thm}
\label{pro:vip4}
For $i\in \mathbb{Z}_4$, the following relation holds on every nonzero finite-dimensional ${\square}_q$-module:
\begin{flalign*}
{\rm {exp}}_q(\mathfrak{n}_{i,i+1})x_{i+2}{\rm {exp}}_q(\mathfrak{n}_{i,i+1})^{-1}=x_{i+2}-x_{i+1}^{-1}+\frac{qx_{i+2}x_ix_{i+1}}{q-q^{-1}}
-\frac{x_ix_{i+2}x_{i+1}}{q(q-q^{-1})}\\
+\frac{q^3x_{i+2}x_{i}^2x_{i+1}^2}{(q-q^{-1})(q^2-q^{-2})}
+\frac{qx_i^2x_{i+2}x_{i+1}^2}{(q-q^{-1})(q^2-q^{-2})}-\frac{q^2x_ix_{i+2}x_{i}x_{i+1}^2}{(q-q^{-1})^2}.
\end{flalign*}
\end{thm}
\begin{proof}
Apply the map ${\phi}$ from Lemma {\ref {lem:m3}} to each side of the equation in Theorem {\ref {pro:vip1}}.
\end{proof}

\section{Acknowledgement}
This paper was written while the author was a graduate student at the University of
Wisconsin-Madison. The author would like to thank his advisor, Paul Terwilliger, for offering
many valuable ideas and suggestions.

As part of computational evidence, the open software SageMath (see \cite{sage}) was used to verify our main results Theorems {\ref {pro:81}}, {\ref {pro:82}} and Theorems {\ref {pro:vip1}}--{\ref {pro:vip4}} on low dimensional irreducible $\square_q$-modules.

\end{document}